\documentclass[12pt]{article}
\usepackage{arxiv}
\usepackage[T1]{fontenc}
\usepackage{mathtools,amsthm,microtype,graphicx,lastpage}
\usepackage[nofontspec]{newtxtext}
\usepackage{newtxmath}
\makeatletter
\renewcommand{\normalsize}{\@setfontsize\normalsize{12}{15}\abovedisplayskip 10pt plus 2pt minus 3pt\belowdisplayskip 10pt plus 2pt minus 3pt\abovedisplayshortskip 3pt plus 2pt\belowdisplayshortskip 6pt plus 2pt minus 2pt}
\renewcommand{\small}{\@setfontsize\small{11}{13.5}}
\renewcommand{\footnotesize}{\@setfontsize\footnotesize{10}{12}}
\makeatother
\normalsize
\usepackage{tikz}
\usetikzlibrary{arrows.meta,calc,positioning}
\usepackage[hidelinks,
  bookmarksnumbered,pdfencoding=auto,psdextra]{hyperref}
\hypersetup{pdftitle={Finite presentations of metabelian groups: effective enumeration via Laurent relations},
  pdfauthor={Achyuth Jayadevan},
  pdfsubject={Effective enumeration of ordinary finite metabelian group presentations}}
\numberwithin{equation}{section}
\newtheorem{theorem}{Theorem}[section]
\newtheorem{proposition}[theorem]{Proposition}
\newtheorem{lemma}[theorem]{Lemma}

\theoremstyle{definition}

\theoremstyle{remark}

\newcommand{\Z}{\mathbb Z}

\newcommand{\R}{\mathbb R}
\newcommand{\N}{\mathbb N}
\newcommand{\nc}[1]{\langle\!\langle #1\rangle\!\rangle}
\newcommand{\gp}[1]{\langle #1\rangle}
\newcommand{\norm}[1]{\lVert #1\rVert}

\newcommand{\im}{\operatorname{im}}
\newcommand{\Met}{\mathcal M}
\newcommand{\laur}{\Z[t_1^{\pm1},\ldots,t_k^{\pm1}]}
\newcommand{\epito}{\twoheadrightarrow}
\allowdisplaybreaks[1]
\title{\bfseries Finite presentations of metabelian groups:\\ effective enumeration via Laurent relations}
\author{\href{https://orcid.org/0009-0008-8745-4078}{\raisebox{-1.2pt}{\includegraphics[height=11pt]{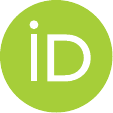}}\hspace{.45em}Achyuth Jayadevan}\\
  \href{https://www.manipal.edu/mit.html}{\raisebox{-1.5pt}{\includegraphics[height=11pt]{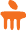}}}\hspace{.4em}Manipal Institute of Technology, MAHE\\
  \href{mailto:achyuth@jayadevan.in}{\textcolor{blue}{\texttt{achyuth@jayadevan.in}}}
}
\date{}
\renewcommand{\headeright}{}
\renewcommand{\undertitle}{}
\makeatletter
\renewcommand{\@maketitle}{\vbox{\hsize\textwidth\linewidth\hsize
    \vskip 0.1in
    \@toptitlebar
    \centering
    {\LARGE\scshape\@title\par}
    \@bottomtitlebar
    \vskip 8pt
    {\renewcommand{\arraystretch}{1.12}\begin{tabular}{c}\bfseries\@author\end{tabular}\par}
    \vskip 16pt}}
\makeatother
\renewcommand{\shorttitle}{\footnotesize Finite presentations of metabelian groups}
\begin{document}
\raggedbottom
\maketitle
\begin{abstract}
For an ordinary finite presentation $P=\langle x_1,\ldots,x_n\mid R\rangle$,
put $G(P)=F_n/\nc{R}$. We construct a primitive-recursive predicate $V$ with
\[
 G(P)''=1\quad\Longleftrightarrow\quad\exists c\in\N:\ V(P,c)=1.
\]
Thus finite presentations of metabelian groups are recursively enumerable,
answering Kourovka Problem~17.124. An effective form of the
Bieri--Strebel covering construction, using signed Laurent relations
and rational separation, gives a family of finitely presented metabelian groups
cofinal under epimorphisms. Products of conjugates of defining relators
witness these epimorphisms. The construction and enumeration theorem
are formalized in Lean~4.
\end{abstract}

\section{The enumeration theorem}
\label{sec:introduction}

Let $F_n=F(x_1,\ldots,x_n)$ and let $R$ be a finite list of words in
$F_n$. We consider ordinary presentations
\[
 P=\gp{x_1,\ldots,x_n\mid R},\qquad
 G(P)=F_n/\nc{R},\qquad
 \Met=\{P:G(P)''=1\}.
\]
The empty generating set is allowed. Encodings with a letter outside
$\{x_1,\ldots,x_n\}$ are rejected.

\begin{theorem}
\label{thm:main}
There is a primitive-recursive predicate $V$ on natural-number
encodings of presentations and finite witnesses such that
\begin{equation}
\label{eq:main}
 P\in\Met
 \quad\Longleftrightarrow\quad
 \exists c\in\N:\ V(P,c)=1.
\end{equation}
In particular, $\Met$ is recursively enumerable.
\end{theorem}

This answers the enumeration question in Kourovka Problem~17.124
\cite{Notebook} and the enumeration and witness questions of Shpilrain
\cite[Problems~4 and~5]{Shpilrain}. The ordinary-presentation
recognition question also occurs in
\cite[Problem~42, arXiv version]{BPR}, with attribution to Miller
\cite{Miller}.

The covering presentations come from Bieri and Strebel
\cite[\S3.5, equations~(3.1)--(3.4), Lemma~3.5 and \S3.6]{BS}.
We make their finite data effectively recognizable and obtain ordinary
finite presentations $P_R$ satisfying
\begin{equation}
\label{eq:strategy}
 P_R''=1,\qquad
 G''=1,\ G\text{ finitely presented}
 \ \Longrightarrow\ \exists P_R\epito G.
\end{equation}
A finite witness for the epimorphism consists of identities in a free
group, each a product of conjugates of defining relators.
These two ingredients give \eqref{eq:main}.

The algorithms of Baumslag, Cannonito and Robinson
\cite[p.~630 and Theorems~2.1, 3.1]{BCR} take finite presentations in
the metabelian variety. Here the input is $F_n/\nc{R}$ in the variety
of all groups. The geometric and module-theoretic ingredients remain
classical; the construction makes their finite data sufficient for
ordinary-presentation enumeration.

For abelian-by-cyclic groups, recursive enumerability of ordinary finite
presentations is already recorded in the proof of
\cite[Lemma~11.4]{GMW}. Theorem~\ref{thm:main} treats arbitrary
metabelian groups. Its certificate gives explicit finite Laurent data
and normal-closure derivations witnessing metabelianity.

\section{Conventions and classical inputs}
\label{sec:classical}

We use right conjugation and right commutators:
\[
 a^w=w^{-1}aw,\qquad [a,b]=a^{-1}b^{-1}ab.
\]
A group $G$ is metabelian when its derived subgroup $G'$ is abelian,
equivalently when $[[a,b],[c,d]]=1$ for all $a,b,c,d\in G$.
Given letters $t_1,\ldots,t_k$, put
\[
 q(v)=t_1^{v_1}\cdots t_k^{v_k}\qquad(v\in\Z^k).
\]
We retain the literal inverse
\[
 q(v)^{-1}=t_k^{-v_k}\cdots t_1^{-v_1}.
\]
Unsubscripted norms are Euclidean.

\begin{lemma}[Bieri--Strebel collection]
\label{lem:collection}
Let $k\ge1$ and let $\mathcal A$ be a finite alphabet, with a designated letter $a_{ij}$
for each $1\le i<j\le k$. For $\rho>0$, let $H_\rho$ have generators
$t_1,\ldots,t_k$ and $\mathcal A$, and relators
\[
 [t_i,t_j]=a_{ij},\qquad
 [a,b^{q(v)}]=1
 \quad(a,b\in\mathcal A,\ v\in\Z^k,\ \norm v<\rho).
\]
Suppose $\rho>2k$ and $u,v\in\Z^k$ satisfy
\begin{equation}
\label{eq:collection}
 \norm u\le \frac{\rho}{2k},\qquad
 \norm v<\rho+\frac1{2k},\qquad
 \norm{u+v}<\rho.
\end{equation}
Then, for all $a,b\in\mathcal A$, the following hold in $H_\rho$:
\[
 [a,b^{q(u)q(v)}]=1,\qquad
 [a,b^{q(u)^{-1}q(v)^{-1}}]=1.
\]
If all ordered commutation relators are imposed, with no radius
restriction, the normal closure of $\mathcal A$ is abelian.
\end{lemma}

These are the ordered and inverse-ordered cases of
\cite[Lemma~3.4b]{BS}; the normal-closure assertion is
\cite[Corollary~3.3]{BS}. The collection argument
\cite[Lemma~3.2]{BS} is valid in every quotient of $H_\rho$.

For the second input, let $A$ be an abelian normal subgroup of a group
$E$ with $E/A\cong\Z^k$. Conjugation gives $A$ a right module structure
over $\laur$, independent of the chosen lifts of the quotient generators.
Let $A^*$ denote the inverse-action module, in which $t_i$ acts as
$t_i^{-1}$ acts on $A$.

\begin{lemma}[Finite signed tameness data]
\label{lem:tameness}
Let $k\ge1$. Suppose $E$ is finitely presented and metabelian, $E/A\cong\Z^k$, and
$A$ is finitely generated as a $\laur$-module. There is a finite family
of nonzero Laurent polynomials
\[
 \lambda=\sum_{u\in S_\lambda} c_{\lambda,u}t^u
 \quad(c_{\lambda,u}\in\Z\setminus\{0\}),
\]
indexed by a signed family $\Lambda=\Lambda_+\sqcup\Lambda_-$, such that
\begin{gather*}
 A(\lambda-1)=0\quad(\lambda\in\Lambda_+),\qquad
 A^*(\lambda-1)=0\quad(\lambda\in\Lambda_-),\\[3pt]
 \forall v\in\R^k\setminus\{0\}\ \exists\lambda\in\Lambda:
       \min_{u\in S_\lambda}v\cdot u>0.
\end{gather*}
\end{lemma}

Apply the necessity direction of \cite[Theorem~A]{BS} and the
centralizer criterion \cite[Lemma~2.6]{BS} to $A$ and $A^*$.
For $A=0$, use the monomials $t_i^{\pm1}$; for $A\ne0$, an
identity-acting polynomial is nonzero. The minus sign changes the
module action; the displayed support remains $S_\lambda$.

\section{Rational separation and the commutation radius}
\label{sec:geometry}

Fix $k\ge1$ and a finite family $(S_\lambda)_{\lambda\in\Lambda}$
of nonempty subsets of $\Z^k$. Consider
\begin{equation}
\label{eq:cover}
 \forall v\ne0\ \exists\lambda\ \forall u\in S_\lambda:
       v\cdot u>0.
\end{equation}
Only rational arithmetic is needed to test this condition.

\begin{lemma}[Rational separation]
\label{lem:margin}
Condition~\eqref{eq:cover} is decidable. If it holds, there is an
integer $r\ge0$ such that, for $\varepsilon=2^{-r}$,
\begin{equation}
\label{eq:margin}
 \forall v\in\R^k,\ \norm v_\infty=1
 \ \Longrightarrow\
 \exists\lambda\ \forall u\in S_\lambda:\ v\cdot u>\varepsilon.
\end{equation}
Whether a proposed $r$ satisfies~\eqref{eq:margin} is decidable by a
primitive-recursive computation on the finite integer data.
\end{lemma}

\begin{proof}
Put $\Sigma=\{v\in\R^k:\norm v_\infty=1\}$. Its faces are
\[
 F_{j,s}=\{v\in[-1,1]^k:v_j=s\},\qquad
 \Sigma=\bigcup_{j=1}^{k}\ \bigcup_{s\in\{-1,1\}}F_{j,s}.
\]
Failure of \eqref{eq:margin} on $F_{j,s}$ is equivalent to the
existence of $(u_\lambda)_\lambda\in\prod_\lambda S_\lambda$ and $v$ satisfying
\begin{equation}
\label{eq:linear}
 v\cdot u_\lambda\le\varepsilon\quad(\lambda\in\Lambda),
 \qquad -1\le v_i\le1,\qquad v_j=s.
\end{equation}
Fourier--Motzkin elimination is the following rational equivalence.
For $\alpha_i x+\beta_i\cdot y\le\gamma_i$, put
$I_\pm=\{i:\pm\alpha_i>0\}$ and $I_0=\{i:\alpha_i=0\}$. Then
\[
 \exists x\ \forall i:\ \alpha_i x+\beta_i\cdot y\le\gamma_i
 \quad\Longleftrightarrow\quad
 \begin{cases}
  \beta_i\cdot y\le\gamma_i &(i\in I_0),\\[4pt]
  \dfrac{\gamma_n-\beta_n\cdot y}{\alpha_n}
  \le\dfrac{\gamma_p-\beta_p\cdot y}{\alpha_p}
       &(n\in I_-,\ p\in I_+).
 \end{cases}
\]
After $k$ eliminations, only rational inequalities remain. At each
stage, rational values of the remaining coordinates admit a rational
choice of the eliminated coordinate. Thus rational and real
feasibility agree. Applying the same test with $\varepsilon=0$
decides \eqref{eq:cover}.

Under \eqref{eq:cover}, $\Lambda\ne\varnothing$ and compactness gives
\[
 0<\mu:=\min_{v\in\Sigma}\max_\lambda
                \min_{u\in S_\lambda}v\cdot u,
 \qquad 2^{-r}<\mu\ \Longrightarrow\ \eqref{eq:margin}.
\]
Hence some $r\in\N$ succeeds. For a supplied $r$, the $k$ elimination
steps, finite choices of supports and exact rational comparisons
are primitive recursive.
\end{proof}

For data passing this margin test put
\begin{equation}
\label{eq:constants}
 C=\frac{2^{-r}}{k},\qquad
 D=1+\sum_\lambda\sum_{u\in S_\lambda}\norm u_1,\qquad
 R=1+2k\bigl(1+D+D^2k\,2^r\bigr).
\end{equation}
Thus $D$ is an integer, $0<C\le1$, and
$R=1+2k(1+D+D^2/C)$. In particular,
\begin{equation}
\label{eq:large-radius}
 R>2k,\qquad D<\frac{R}{2k},\qquad R>\max\{C,D^2/C\}.
\end{equation}
By homogeneity and $\norm v_\infty\ge1/k$ for $\norm v=1$,
\eqref{eq:margin} implies that for each Euclidean unit vector $v$
some support satisfies $v\cdot u>C$ throughout.

\begin{lemma}[Radial propagation]
\label{lem:shrink}
Put $\delta=\min\{C/4,1/(4k)\}$. If $\rho\ge R$ and
$\rho\le\norm v<\rho+\delta$, there is a support $S_\lambda$ such that
\[
 \norm{v+u}<\rho \qquad(u\in S_\lambda).
\]
\end{lemma}

\begin{proof}
Put $s=\norm v$. The margin in direction $-v/\norm v$ gives
$\lambda$ with $v\cdot u\le-Cs$ for all $u\in S_\lambda$.
For each such $u$, $\norm u<D$, and
\begin{align*}
 \norm{v+u}^2
 &=s^2+2v\cdot u+\norm u^2\\
 &<s^2-2Cs+D^2\\
 &<(\rho+\delta)^2-2C(\rho+\delta)+D^2\\
 &=\rho^2-2(C-\delta)\rho+\delta^2-2C\delta+D^2\\
 &\le\rho^2-\tfrac32 C\rho+D^2
 <\rho^2.
\end{align*}
Indeed $s\ge\rho\ge R>C$, $\delta\le C/4$ and $\rho>D^2/C$.
The same $\lambda$ works for every $u\in S_\lambda$.
\end{proof}

\section{Metabelian presentations from Laurent relations}
\label{sec:covers}

A \emph{cover datum} consists of $k\ge1$, a finite nonempty alphabet
$\mathcal A$ with designated letters $a_{ij}$, a finite family of signed
nonzero Laurent polynomials, and an integer $r$ passing the margin test.
Write $\Lambda_+$ and $\Lambda_-$ for the two signed families, and
fix an order on every support. Using the radius~\eqref{eq:constants},
define $P_R$ on the alphabet $\{t_1,\ldots,t_k\}\cup\mathcal A$ by
\begin{alignat}{2}
 &[t_i,t_j]=a_{ij}
   &\qquad &(1\le i<j\le k), \label{eq:cover-t}\\[3pt]
 &[a,b^{q(v)}]=1
   &\qquad &(a,b\in\mathcal A,\ \norm v<R), \label{eq:cover-short}\\[3pt]
 &a=\prod_{u\in S_\lambda}(a^{c_{\lambda,u}})^{q(u)}
   &\qquad &(a\in\mathcal A,\ \lambda\in\Lambda_+), \label{eq:cover-plus}\\[3pt]
 &a=\prod_{u\in S_\lambda}(a^{c_{\lambda,u}})^{q(u)^{-1}}
   &\qquad &(a\in\mathcal A,\ \lambda\in\Lambda_-). \label{eq:cover-minus}
\end{alignat}
Products follow the chosen support order; the minus case uses the
literal inverse $q(u)^{-1}$. The relator set is finite since
\[
 B_R=\{v\in\Z^k:\norm v<R\}
 =\left\{v\in[-R,R]^k\cap\Z^k:\sum_i v_i^2<R^2\right\}.
\]

\begin{proposition}[Metabelianity]
\label{prop:sound}
Every group $P_R$ constructed from an accepted cover datum is metabelian.
\end{proposition}

\begin{proof}
In a group satisfying \eqref{eq:cover-t}, \eqref{eq:cover-plus} and
\eqref{eq:cover-minus}, write
\[
 \mathcal C(\rho):\quad
 [a,b^{q(v)}]=1\quad
 (a,b\in\mathcal A,\ v\in\Z^k,\ \norm v<\rho).
\]
We prove $\mathcal C(\rho)\Rightarrow\mathcal C(\rho+\delta)$ for $\rho\ge R$.

Only vectors $v$ with $\rho\le\norm v<\rho+\delta$ need consideration.
Choose $\lambda$ by Lemma~\ref{lem:shrink}. Then, for every
$u\in S_\lambda$,
\begin{equation}
\label{eq:collection-app}
 \norm u<D<\rho/(2k),\qquad
 \norm v<\rho+1/(2k),\qquad
 \norm{u+v}<\rho.
\end{equation}

If $\lambda$ has sign $+$, Lemma~\ref{lem:collection} and
\eqref{eq:collection-app} give
\[
 [a,b^{q(u)q(v)}]=1\qquad(u\in S_\lambda).
\]
Writing $C_H(a)$ for the centralizer of $a$ in the ambient group,
relation~\eqref{eq:cover-plus} therefore yields
\[
 b^{q(v)}
 =\prod_{u\in S_\lambda}(b^{c_{\lambda,u}})^{q(u)q(v)}
 \in C_H(a).
\]
If $\lambda$ has sign $-$, the inverse-word assertion of
Lemma~\ref{lem:collection}, with the alphabet letters interchanged, gives
\[
 [b,a^{q(u)^{-1}q(v)^{-1}}]=1\qquad(u\in S_\lambda).
\]
Now \eqref{eq:cover-minus} yields
\[
 a^{q(v)^{-1}}
 =\prod_{u\in S_\lambda}(a^{c_{\lambda,u}})^{q(u)^{-1}q(v)^{-1}}
 \in C_H(b).
\]
Conjugating this commutation identity by $q(v)$ gives
$[a,b^{q(v)}]=1$. Thus both signs prove the propagation step.

In $P_R$, relation~\eqref{eq:cover-short} supplies $\mathcal C(R)$, so
\[
 \mathcal C(R+n\delta)\quad(n\in\N),\qquad
 \bigcup_{n\ge0}\{v\in\Z^k:\norm v<R+n\delta\}=\Z^k.
\]
All ordered commutators vanish. Lemma~\ref{lem:collection} and
\eqref{eq:cover-t} imply, for $A_R=\nc{\mathcal A}$,
\[
 A_R'=1,\qquad (P_R/A_R)'=1,\qquad
 P_R'\le A_R,\qquad P_R''=1.\qedhere
\]
\end{proof}

\section{Cofinality of the covers}
\label{sec:cofinality}

\begin{lemma}[Central extensions]
\label{lem:central}
If $1\to K\to E\to G\to1$ is a central extension, $G$ is finitely
presented, and $K$ is free abelian of finite rank, then $E$ is finitely
presented.
\end{lemma}

\begin{proof}
Let $G=\langle x_1,\ldots,x_m\mid r_1,\ldots,r_\ell\rangle$.
Choose lifts $y_i\in E$ and a basis $z_1,\ldots,z_s$ of $K$; write
\[
 r_j(y)=\prod_{h=1}^{s}z_h^{b_{jh}},\qquad b_{jh}\in\Z.
\]
For alphabets $Y=\{y_1,\ldots,y_m\}$ and $Z=\{z_1,\ldots,z_s\}$, set
\begin{align*}
 \mathcal R_0
 &=\{[z_h,z_{h'}]:1\le h<h'\le s\}
   \cup\{[z_h,y_i]:1\le h\le s,\ 1\le i\le m\},\\[3pt]
 \mathcal R_1
 &=\left\{r_j(y)\left(\prod_{h=1}^{s}z_h^{b_{jh}}\right)^{-1}
                    :1\le j\le\ell\right\}.
\end{align*}
The lifts induce an epimorphism
\[
 \varphi:\widetilde E:=
 \langle Y\sqcup Z\mid\mathcal R_0\cup\mathcal R_1\rangle
 \epito E.
\]
If $w\in\ker\varphi$, its image in $G$ is trivial. Express its
$Y$-word as a product of conjugates of $r_j^{\pm1}$. The relations
$\mathcal R_1$ and centrality relations $\mathcal R_0$ give
\[
 w=\prod_{h=1}^{s}z_h^{n_h}\quad\text{in }\widetilde E,
 \qquad
 1=\varphi(w)=\prod_{h=1}^{s}z_h^{n_h}\quad\text{in }K.
\]
Independence of the basis implies $n_1=\cdots=n_s=0$. Hence
$\ker\varphi=1$ and $\widetilde E\cong E$.
\end{proof}

\begin{proposition}[Quotient cofinality]
\label{prop:cofinal}
Every finitely presented metabelian group $G$ is an epimorphic image
of some certified cover $P_R$.
\end{proposition}

\begin{proof}
Put $A=G'$ and $Q=G/A$. Choose a surjection
$\pi:\Z^k\epito Q$ with $k\ge1$ and form the pullback
\[
 E=\{(g,z)\in G\times\Z^k:gA=\pi(z)\}.
\]
Both projections are surjective. Their kernels give exact sequences
\begin{equation}
\label{eq:pullback-sequences}
 1\longrightarrow A\longrightarrow E\longrightarrow\Z^k\longrightarrow1,
 \qquad
 1\longrightarrow\ker\pi\longrightarrow E\longrightarrow G\longrightarrow1.
\end{equation}
The second kernel consists of $(1,z)$ with $z\in\ker\pi$, and is
central and free abelian of finite rank. Lemma~\ref{lem:central}
makes $E$ finitely presented. As a subgroup of the metabelian group
$G\times\Z^k$, it is metabelian.

For a finite generating set $g_1,\ldots,g_m$ of $G$,
\[
 G'=\nc{[g_i,g_j]:i<j},\qquad
 A=\sum_{i<j}[g_i,g_j]\,\Z Q.
\]
The second equality uses additive module notation for the abelian
group $A$. Pullback along $\pi$ preserves these module generators.

Choose lifts $t_i\in E$ of the standard basis of $\Z^k$. Take a finite
module-generating alphabet $\mathcal A$ for $A$, including letters
representing every $[t_i,t_j]$. A letter representing the identity
may be added to ensure that the alphabet is nonempty.
Lemma~\ref{lem:tameness} supplies signed Laurent identities on this
module. Lemma~\ref{lem:margin} supplies a finite margin certificate.

The selected elements satisfy \eqref{eq:cover-t} by construction,
\eqref{eq:cover-short} because $A'=1$, and
\eqref{eq:cover-plus}--\eqref{eq:cover-minus} by the module identities.
They induce $\psi:P_R\to E$. If $B$ is the image of $\mathcal A$, then
\[
 A=\langle b^{q(v)}:b\in B,\ v\in\Z^k\rangle
 \le\im\psi,
 \qquad
 E=A\langle t_1,\ldots,t_k\rangle=\im\psi.
\]
Thus $\psi$ is surjective, and $P_R\epito E\epito G$.
\end{proof}

\begin{figure}[tb]
\centering
\begin{tikzpicture}[font=\small,line width=.45pt,
  >={Latex[length=1.6mm,width=1.1mm]},
  every node/.style={inner sep=3pt}]
 \node (p) at (-2.6,0) {$P_R$};
 \node (e) at (0,0) {$E$};
 \node (g) at (3,0) {$G$};
 \node (z) at (0,-1.6) {$\Z^k$};
 \node (q) at (3,-1.6) {$Q=G/G'$};
 \draw[->,shorten <=3pt,shorten >=3pt] (p)--(e);
 \draw[->,shorten <=3pt,shorten >=3pt]
   (e)--node[above=3pt] {$\mathrm{pr}_1$}(g);
 \draw[->,shorten <=3pt,shorten >=3pt]
   (e)--node[left=3pt] {$\mathrm{pr}_2$}(z);
 \draw[->,shorten <=3pt,shorten >=3pt] (g)--(q);
 \draw[->,shorten <=3pt,shorten >=3pt]
   (z)--node[below=3pt] {$\pi$}(q);
\end{tikzpicture}
\caption{The pullback $E=G\times_Q\Z^k$ and the epimorphism from $P_R$.
All arrows are surjective.}
\label{fig:pullback}
\end{figure}

\section{Epimorphisms and normal closures}
\label{sec:certificates}

For $P=\gp{X\mid R}$ and $w\in F(X)$,
\begin{equation}
\label{eq:derivation}
 w\in\nc{R}
 \quad\Longleftrightarrow\quad
 w=\prod_{\nu=1}^{m}h_\nu^{-1}r_{j_\nu}^{\epsilon_\nu}h_\nu
 \text{ in }F(X),
 \qquad \epsilon_\nu\in\{-1,1\}.
\end{equation}
Here the right side means existence of $m\ge0$, $h_\nu\in F(X)$,
$r_{j_\nu}\in R$ and the indicated signs. A specified identity is
checked by free reduction.

\begin{lemma}
\label{lem:epi}
Let $H=\gp{y_1,\ldots,y_\ell\mid S}$ and
$G=\gp{x_1,\ldots,x_n\mid R}$. Then
\[
 \exists\,f:H\epito G
 \quad\Longleftrightarrow\quad
 \begin{gathered}
 \exists u_1,\ldots,u_\ell\in F(X),\
         v_1,\ldots,v_n\in F(Y):\\
 s(u_1,\ldots,u_\ell)\in\nc{R}\quad(s\in S),\\
 v_j(u_1,\ldots,u_\ell)x_j^{-1}\in\nc{R}\quad(1\le j\le n).
 \end{gathered}
\]
Together with products \eqref{eq:derivation} for these finitely many
normal-closure memberships, the words $u_i,v_j$ have a
primitive-recursive checking predicate.
\end{lemma}

\begin{proof}
The first family of identities induces
\[
 f:H\longrightarrow G,\qquad y_i\longmapsto u_i\nc{R}.
\]
The second gives $f(v_j)=x_j\nc{R}$, hence $\im f=G$.
Conversely, choose words representing the images and preimages of
the generators under an epimorphism. Its defining identities lie
in $\nc{R}$ and therefore admit \eqref{eq:derivation}.
Substitution and free reduction of finite words, and the finite
integer and index checks in \eqref{eq:derivation}, are primitive recursive.
\end{proof}

\begin{proof}[Proof of Theorem~\ref{thm:main}]
Write $\mathsf A(d,r)$ for the margin condition
\eqref{eq:margin} on finite Laurent data $d$, and $P_{d,r}$ for
the presentation \eqref{eq:cover-t}--\eqref{eq:cover-minus}
at the radius \eqref{eq:constants}.
Let $\mathsf E(H,P;\eta)$ check the words and products in
Lemma~\ref{lem:epi}. For a decoded witness $c=(d,r,\eta)$, set
\begin{equation}
\label{eq:predicate}
\begin{split}
 V(P,c)=1\quad\Longleftrightarrow\quad&
 \mathsf{WF}(P)\ \land\\
 &\bigl(n(P)=0\ \lor\
 [\mathsf A(d,r)\land\mathsf E(P_{d,r},P;\eta)]\bigr),
\end{split}
\end{equation}
where $\mathsf{WF}$ checks the generating alphabet and $n(P)$ is its size.
All finite syntax checks are included in $\mathsf A$ and $\mathsf E$.
The radius is an integer expression in the data, and
\[
 \{v\in\Z^k:\norm v<R\}
 =
 \left\{v\in[-R,R]^k\cap\Z^k:\sum_i v_i^2<R^2\right\}.
\]
Consequently, constructing $P_{d,r}$ uses a bounded finite list.
Lemma~\ref{lem:margin} and Lemma~\ref{lem:epi} show that $V$
is primitive recursive.

For $n(P)>0$, soundness is the implication
\[
 V(P,c)=1
 \ \Longrightarrow\
 P_{d,r}''=1,\quad P_{d,r}\epito G(P)
 \ \Longrightarrow\ G(P)''=1,
\]
by Proposition~\ref{prop:sound}. For completeness,
Proposition~\ref{prop:cofinal} and Lemma~\ref{lem:margin} give
\[
 G(P)''=1
 \ \Longrightarrow\
 \exists d,r:\ \mathsf A(d,r),\quad P_{d,r}\epito G(P).
\]
Lemma~\ref{lem:epi} supplies $\eta$. If $n(P)=0$ and $P$ is
well formed, $G(P)=1$, as required by the first branch of
\eqref{eq:predicate}. This proves \eqref{eq:main}.
Finally, projection of the primitive-recursive relation
$\{(P,c):V(P,c)=1\}$ gives the asserted recursive enumeration.
\end{proof}

\section*{Foundation model assistance for Formalisation}
Theorem~\ref{thm:main} and its algebraic ingredients are formalised
in Lean~4 with assistance from GPT-Astra through a custom LangGraph harness.

\end{document}